\documentclass[%
version=last,%
a5paper,%
12pt,%
bibliography=totoc,
twoside,%
draft=true,%
titlepage=true,%
DIV=12,%
headinclude=false,%
pagesize]%
{scrartcl}

\usepackage{scrpage2}
\clearscrheadings
\ofoot{\pagemark}
\refoot{\leftmark}
\lofoot{\rightmark}

\usepackage{moredefs,lips}

\usepackage[neverdecrease]{paralist}
\usepackage{hfoldsty,verbatim}

\usepackage{relsize}

\usepackage{mparhack}

\usepackage{amsmath,amsthm,amssymb,url,upgreek,bm}
\newcommand{\val}{\operatorname{val}}
\newcommand{\vr}{\mathfrak O}
\newcommand{\invr}{\mathord{\in}\,\vr}
\newcommand{\maxi}{\mathfrak M}
\newcommand{\inmaxi}{\mathord{\in}\,\maxi}
\newcommand{\F}{\mathbb F}
\newcommand{\alg}{^{\mathrm{alg}}}
\newcommand{\autom}{^{\sigma}}
\newcommand{\starred}[1]{#1^*}
\newcommand{\str}[1]{\mathfrak{#1}}

\DeclareMathOperator{\diagram}{diag}
\newcommand{\diag}[1]{\diagram(#1)}
\newcommand{\proves}{\vdash}

\renewcommand{\vec}[1]{\bm{#1}}

\renewcommand{\phi}{\varphi}
\renewcommand{\theta}{\vartheta}
\renewcommand{\leq}{\leqslant}
\renewcommand{\geq}{\geqslant}
\renewcommand{\setminus}{\smallsetminus}

\newtheorem{theorem}{Theorem}
\newtheorem{lemma}[theorem]{Lemma}
\newtheorem{fact}[theorem]{Fact}
\newtheorem{porism}[theorem]{Porism}

\theoremstyle{definition}

\newcommand{\theory}[1]{\mathsf{#1}}
\newcommand{\fav}{\theory{FAV}}
\newcommand{\acfa}{\theory{ACFA}}
\newcommand{\dcf}{\theory{DCF}}
\newcommand{\acf}{\theory{ACF}}
\newcommand{\acvf}{\theory{ACVF}}

\newcommand{\df}{\theory{DF}}

\newcommand{\included}{\subseteq}

\newcommand{\lto}{\Rightarrow}
\newcommand{\liff}{\Leftrightarrow}
\newcommand{\Forall}[1]{\forall{#1}\;}
\newcommand{\Exists}[1]{\exists{#1}\;}

\usepackage{mathrsfs}
\newcommand{\sig}{\mathscr S}

\begin{document}
\title{Fields with automorphism and valuation}
\author{\"Ozlem Beyarslan \and Daniel Max Hoffmann
  \and G\"onen\c c Onay \and David Pierce}
\date{October 10, 2017}
\lowertitleback{\tableofcontents}

\maketitle


\setcounter{section}{0}
\section{Introduction}

The question has been of interest for decades:
If a model-complete theory $T$
is augmented with axioms for an automorphism $\sigma$ of its models,
does the resulting theory $T_{\sigma}$ have a model-companion?

It does have,
when $T$ is the theory $\acf$ of algebraically closed fields.
The model companion of $T_{\sigma}$ in this case is $\acfa$,
studied by Macintyre \cite{MR99c:03046}
and Chatzidakis and Hrushovski \cite{MR2000f:03109}
and others.
However, $T_{\sigma}$ is not companionable when $T$ is $\acfa$ itself
\cite{MR2138334} .

A more general result, established by Kikyo \cite{MR1791373},
is that if $T_{\sigma}$ is companionable, and $T$ is dependent,
then $T$ must also be stable.
In particular then, $T_{\sigma}$ cannot be companionable
when $T$ is the theory $\acvf$ of algebraically closed valued fields
in the signature of fields with a predicate for a valuation ring.
Note that an automorphism $\sigma$ of a valued field
induces an automorphism $\sigma_v$ of the the value group $\Gamma$
and an automorphism $\bar{\sigma}$ of the residue field.
Then $T_{\sigma}$ is companionable
when $T$ is the model companion of the theory
of any of the following classes of valued fields:
\begin{compactenum}[1)]
\item
valued $D$-fields \cite{MR2002d:03066},
\item
isometric valued difference fields, where $\sigma_v(\gamma)=\gamma$
for all $\gamma$ in $\Gamma$ \cite{MR2325101,MR2749570},
\item
  contractive valued difference fields,
  where $\sigma_v(\gamma)>n\gamma$ for all positive $\gamma$ in $\Gamma$
  and $n$ in $\upomega$ \cite{MR2725200},
\item multiplicative valued fields,
where $\sigma_v(\gamma)=\rho\gamma$ for all $\gamma$ in $\Gamma$,
for a certain constant $\rho$
\cite{MR2963021}.
\end{compactenum}
Moreover,
\begin{compactenum}[1)]\setcounter{enumi}{4}
\item\sloppy
  $T_{\sigma}\cup T_v$ is companionable
	when $T$ is $\acvf$ and
  $T_v$ is a companionable theory of ordered abelian groups
  equipped with an automorphism
   \cite{MR3397448} (in this case $(\Gamma,\sigma_v)\models T_v$).
	\end{compactenum}
The corresponding model companion in each of the five cases
satisfies an analogue
of Hensel's lemma for $\sigma$-polynomials
(see \cite[Definition 4.2]{MR2725200}).

In this paper we consider the theory $\fav$
of a valued field equipped with an automorphism of the field alone.
There is no required interaction of valuation and the automorphism:
the automorphism need not fix the valuation ring (setwise).
The similar case of a differential field with an automorphism of the field alone
was treated in \cite{MR2114160}.
Our main theorem,
Theorem \ref{thm:main},
is that $\fav$ has a model companion, $\fav^*$.

There is an obvious candidate for $\fav^*$,
since $\fav$ is included in the union of two model-complete theories,
namely $\acfa$ and $\acvf$.
However, we show, as Theorem \ref{thm:opt},
that $\acfa\cup\acvf$ is \emph{not} model complete.

Our paper is organized as follows.
In \S\ref{sect:fav} we give axioms of $\fav$.
In \S\ref{sect:mc}
preliminaries about companionable theories are explained.
Then in \S\ref{sect:df},
Theorem \ref{thm:acfa}
establishes a geometric axiomatization of $\acfa$.
Using this,
in \S\ref{sect:mod-comp}
we prove Theorems \ref{thm:main} and \ref{thm:opt}.

We thank the Nesin Mathematics Village
in \c Sirince, Sel\c c\"uk, \.Izmir, Turkey,
for hosting the workshop in July, 2016,
where these results were worked out;
and Rahim Moosa and Thomas Scanlon, for organizing the workshop;
and Moshe Kamensky and Piotr Kowalski, who also participated.

\section{Fields with an automorphism and a valuation}\label{sect:fav}

\begin{sloppypar}  
  A signature sufficient for a first-order axiomatization
  of fields with an automorphism and a valuation
  is the signature $\{+,-,\times,0,1\}$ of fields,
  augmented with
  \begin{compactenum}[1)]
  \item
    a singulary operator ${}\autom$ for the automorphism and
  \item
    a singulary predicate $\invr$ for membership in the valuation ring.
  \end{compactenum}
  We shall write the last two symbols \emph{after} their arguments.
  The fields with an automorphism and a valuation
  are then axiomatized by the field axioms, along with axioms
  \begin{align*}
    (x+y)\autom&=x\autom+y\autom,&
    (x\cdot y)\autom&=x\autom\cdot y\autom,&	
    \Exists yy\autom=x
  \end{align*}
  for a surjective endomorphism
  (which for a field \emph{is} an automorphism),
  and axioms
  \begin{gather*}
    0\in\vr,\\
    x\in\vr\land y\in\vr\lto -x\in\vr\land x+y\in\vr\land x\cdot y\in\vr,\\
    \Exists y(x\notin\vr\lto x\cdot y=1\land y\in\vr)
  \end{gather*}
  for a valuation ring.
  It will be our habit, as here, to suppress outer universal quantifiers.
  For convenience, we introduce a singulary predicate $\inmaxi$ 
  for membership in the unique maximal ideal of the valuation ring.
  This means requiring
  \begin{equation*}
    x   \in\maxi\liff\Exists y\bigl(x=0\lor(x\cdot y=1\land y\notin\vr)\bigr),
  \end{equation*}
  or equivalently
  \begin{equation}\label{eqn:notin-m}
    x\notin\maxi\liff\Exists y             (x\cdot y=1\land y   \in\vr).
  \end{equation}
  Because both the new predicate and its negation can thus be given existential definitions,
  use of the predicate does not affect the existence of a \emph{model-companion} of the theory being axiomatized
  \cite[Lem.~1.1, p.\ 427]{MR2505433}.
  Let us denote this theory by $\fav$;
  officially, its signature is
  \begin{equation*}
    \{+,-,\times,0,1,{}\autom,\invr,\inmaxi\}.
  \end{equation*}
  For later use, we note that the identity
  \begin{equation}\label{eqn:notin-o}
    x\notin\vr\liff\Exists y(x\cdot y=1\land y\in\maxi)
  \end{equation}
  holds in $\fav$.
\end{sloppypar}
For a valuation as such,
we can introduce a new sort having signature $\{+,0,\infty,>\}$,
so that the valuation is a \emph{surjective} function $\val$ from the original sort to the new sort
that satisfies also
\begin{gather}\label{eqn:val+}
	\val(x)+\val(y)=\val(x\cdot y),\\\notag
	0=\val(1),\\\notag
	\infty=\val(0),\\\notag
	\val(x)>\val(y)\liff\Exists z(y\cdot z=1\land x\cdot z\in\maxi).
\end{gather}
These rules ensure that the new sort is an ordered additive abelian group---%
the \emph{value group}---%
with an additional element $\infty$ that is greater than all others,
and
\begin{equation*}
  \infty+x=\infty=x+\infty.
\end{equation*}
Also, $\val$ restricts to
a homomorphism from the multiplicative group of units of the field 
onto the value group,
and the kernel of this homomorphism is $\vr\setminus\maxi$,
which is the group of units of the valuation ring $\vr$.
Moreover,
\begin{align*}
	\val(x)\geq0&\liff x\in\vr,\\
	\val(x)>0&\liff x\in\maxi, \\
	\val(x)<0 &\liff x\not\in\vr.
\end{align*}
As with the maximal ideal $\maxi$,
so with the value group,
its official status does not matter
for our purposes.
Officially we shall not use the value group,
and so we may write a typical model of $\fav$ as $(K,\sigma,\vr)$.
However, the value group may be useful for thinking things through.

\section{Model-companions and -completions}\label{sect:mc}

The \textbf{(Robinson) diagram} of a structure $\str A$ in a signature $\sig$
is the theory $\diag{\str A}$, in the signature $\sig(A)$ (where $A$ is the domain of $\str A$),
of structures in which $\str A$ embeds.
This means $\diag{\str A}$ is axiomatized by 
all of the quantifier-free sentences of $\sig$ 
with parameters from (the underlying set $A$ of) $\str A$
that are true in $\str A$.
Thus
$\diag{\str A}$ is also axiomatized simply 
by the atomic and negated atomic sentences
of $\sig(A)$ that are true in $\str A$.

When it exists, a \textbf{model-companion} of a theory $T_0$
is a theory $T_1$ in the same signature such that
\begin{compactenum}[1)]
\item
for each $i$,
every model of $T_i$ embeds in a model of $T_{1-i}$,
that is, $T_{1-i}\cup\diag{\str A}$ is consistent
whenever $\str A$ is a model of $T_i$; and
\item
$T_1$ is \textbf{model-complete,} that is,
$T_1\cup\diag{\str A}$ is complete whenever $\str A\models T_1$.
\end{compactenum}
The model-companion of a theory is unique when it exists.
It was ``introduced by Barwise, Eklof, Robinson, and Sabbagh in 1969''
\cite[p.\ 609]{MR91c:03026},
these four logicians being known collectively as Eli Bers
\cite[p.\ 410]{MR94e:03002}.
The model-companion generalizes an earlier notion of Robinson
\cite[\S4.3, p.\ 109]{MR0153570}:
the theory $T_1$ is a
\textbf{model-completion} of $T_0$
in case $T_0\included T_1$
and $T_1\cup\diag{\str A}$ is consistent and complete 
whenever $\str A\models T_0$.
If $T_0$ has the model-companion $T_1$,
then $T_1$ is a model-completion of $T_0$ just in case
$T_0$ has the \textbf{amalgamation property,}
that is, two models having a common submodel have a common supermodel
(this is an exercise in Hodges \cite[\S8.4, exer.\ 9, p.\ 390]{MR94e:03002}
attributed to Eli Bers \cite[Lem.\ 2.1, p.\ 254]{MR0277372}).

We say that a theory is \textbf{inductive}
if every union of a chain of models is a model.
Robinson's name for such a theory was \textbf{$\sigma$-persistent;}
but since we are already using the symbol $\sigma$ for a field automorphism,
we prefer the simpler term for the kind of theory in question.
By the Chang--\L o\'s--Suszko Theorem \cite[6.5.9, p.\ 297]{MR94e:03002},
A theory $T$ is inductive if and only if it is precisely the theory $T_{\forall\exists}$
axiomatized by the $\forall\exists$ (or $\forall_2$) consequences of $T$.

By a \textbf{system} we shall mean a (finite)
conjunction of atomic and negated atomic formulas.
For theories,
having a model-companion or -completion
means having an appropriate condition
for when systems over a given model
have solutions in a larger model.
We recall first Robinson's equivalent formulation
of when inductive theories have model-completions;
we review also the proof,
for the sake of the variations that we shall state and use.

\begin{theorem}[Robinson \protect{\cite[\S5.5, p.\ 128]{MR0153570}}]\label{thm:Rob}
For an inductive theory $T$ to have a model-completion,
a sufficient and necessary condition is that,
for every system $\phi(\vec x,\vec z)$
in the signature of $T$,
there is a quantifier-free formula $\theta(\vec x,\vec y)$ in that signature
such that, for all models $\str M$ of $T$,
for all tuples $\vec a$ of parameters from $M$
having the same length as $\vec x$,
the following conditions are equivalent:
\begin{compactenum}[(i)]
\item\label{item:i}
$\phi(\vec a,\vec z)$ 
is soluble in some model of $T\cup\diag{\str M}$,
\item\label{item:ii}
$\theta(\vec a,\vec y)$ is soluble in $\str M$ itself.
\end{compactenum}
When such $\theta$ do exist,
then the model-completion of $T$ 
is the theory $T^*$ axiomatized by the sentences
\begin{equation}\label{eqn:Rob}
\Exists{\vec y}\theta(\vec x,\vec y)\lto\Exists{\vec z}\phi(\vec x,\vec z),
\end{equation}
along with axioms for $T$ itself.
\end{theorem}

\begin{sloppypar}
  The sentence \eqref{eqn:Rob} is equivalent to the $\forall\exists$
  sentence $\Exists{\vec z}\bigl(\theta(\vec x,\vec y)\lto\phi(\vec x,\vec z)\bigr)$,
  outer universal quantifiers being suppressed.
\end{sloppypar}

\begin{proof}[Proof of Robinson's theorem.]
For the necessity of the given condition,
suppose $T$ has the model-completion $T^*$.
For every system $\phi(\vec x,\vec z)$ in the signature of $T$,
for every model $\str M$ of $T$,
for every tuple $\vec a$ of parameters from $M$
such that \eqref{item:i} holds,
since every model of $T$ embeds in a model of $T^*$,
we can conclude from the completeness of $T^*\cup\diag{\str M}$
that
\begin{equation*}
T^*\cup\diag{\str M}\proves\Exists{\vec z}\phi(\vec a,\vec z).
\end{equation*}
By Compactness and the Lemma on Constants \cite[2.3.2, p.\ 43]{MR94e:03002},
there is a quantifier-free formula $\theta_{(\str M,\vec a)}(\vec x,\vec y)$
of the signature of $T$ such that
\begin{equation}\label{eqn:T^*}
T^*\proves\Exists{\vec y}\theta_{(\str M,\vec a)}(\vec x,\vec y)\lto\Exists{\vec z}\phi(\vec x,\vec z).
\end{equation}
By Compactness again,
for the given system $\phi$,
there is a disjunction $\theta$ of finitely many of the formulas $\theta_{(\str M,\vec a)}$
such that
for every model $\str M$ of $T$,
for every tuple $\vec a$ of parameters from $M$
having the length of $\vec x$,
if \eqref{item:i}, then \eqref{item:ii}.
If conversely \eqref{item:ii},
then $\phi(\vec a,\vec z)$ is soluble in every model of $T^*\cup\diag{\str M}$,
by \eqref{eqn:T^*};
but such a model is a model of $T\cup\diag{\str M}$,
and so \eqref{item:i} holds.

For the sufficiency of Robinson's condition,
we first show that every model $\str M$ of $T$ 
embeds in a model of the theory $T^*$ having the axioms \eqref{eqn:Rob}
in addition to those of $T$.
Here we shall use \eqref{item:ii} implies \eqref{item:i},
but not the converse.
\begin{inparadesc}
\item[For every] system $\phi(\vec x,\vec z)$ in the signature of $T$,
\item[for all] $\vec a$ and $\vec b$ from $M$ such that $\str M\models\theta(\vec a,\vec b)$,
\item[for some] model $\str N$ of $T\cup\diag{\str M}$,
\end{inparadesc}
the sentence
\begin{equation*}
\Exists{\vec z}\phi(\vec a,\vec z)
\end{equation*}
is true in $\str N$.
This sentence being existential
and thus preserved in larger models,
by Zorn's Lemma and inductivity of $T$, 
we can move the last of the three bold quantifiers to the front:
in some model $\str M'$ of $T\cup\diag{\str M}$,
for all systems $\phi(\vec x,\vec z)$,
for all $\vec a$ and $\vec b$ from $M$,
the sentence
\begin{equation*}
\Exists{\vec z}\bigl(\theta(\vec a,\vec b)\lto\phi(\vec a,\vec z)\bigr).
\end{equation*}
is true in $\str N$.
Now we can form the chain
\begin{equation*}
\str M\included\str M'\included\str M''\included\cdots,
\end{equation*}
whose limit is a model of $T^*$.
Thus $T^*\cup\diag{\str M}$ is consistent.

We now show $T^*\cup\diag{\str M}$ is complete
by induction on the complexity of sentences.
The theory is complete with respect to existential sentences,
namely $\exists_1$ sentences,
since \eqref{item:i} implies \eqref{item:ii}.
Indeed, suppose the sentence $\Exists{\vec z}\phi(\vec a,\vec z)$ 
is true in some model of $T^*\cup\diag{\str M}$,
where $\phi$ is quantifier-free in the signature of $T$,
and $\vec a$ is from $M$.
Since $\phi$ is a disjunction of systems,
it is enough to assume $\phi$ itself is a system.
Since $T\included T^*$,
we have \eqref{item:i}
and therefore \eqref{item:ii}.
Since the formula $\theta$ here is quantifier-free,
the sentence $\theta(\vec a,\vec b)$ belongs to $\diag{\str M}$
for some $\vec b$ from $M$.
Since the sentence \eqref{eqn:Rob} is an axiom of $T^*$,
we conclude
\begin{equation*}
T^*\cup\diag{\str M}\proves\Exists{\vec y}\phi(\vec a,\vec y).
\end{equation*}
Thus $T^*\cup\diag{\str M}$
is complete with respect to $\exists_1$ sentences.

Suppose now that 
for some positive integer $n$,
for all models $\str M$ of $T$,
the theory $T^*\cup\diag{\str M}$ 
is complete with respect to $\exists_n$ sentences.
For an arbitrary model $\str M$ of $T$,
let $\phi(\vec x,\vec z)$ be an $\forall_n$ formula,
and let $\vec a$ be a tuple of parameters from $M$
such that the $\exists_{n+1}$ sentence $\Exists{\vec z}\phi(\vec a,\vec z)$
is true in some model $\str N$ of $T^*\cup\diag{\str M}$.
Then for some $\vec c$ from $N$,
the sentence $\phi(\vec a,\vec c)$ is true in $\str N$.
Since $\str N\models T$,
by inductive hypothesis we have
\begin{equation*}
T^*\cup\diag{\str N}\proves\Exists{\vec z}\phi(\vec a,\vec z).
\end{equation*}
By Compactness,
there is a quantifier-free formula $\psi(\vec x,\vec y)$
such that
\begin{align*}
\str N&\models\Exists{\vec y}\psi(\vec a,\vec y),&
T^*&\proves\Exists{\vec y}\psi(\vec x,\vec y)\lto\Exists{\vec z}\phi(\vec x,\vec z).	
\end{align*}
Since again $\str N$ is a model of $T^*\cup\diag{\str M}$,
which is complete with respect to existential sentences,
we can conclude
\begin{equation*}
T^*\cup\diag{\str M}\proves\Exists{\vec z}\phi(\vec a,\vec z).
\end{equation*}
Thus $T^*\cup\diag{\str M}$ is complete with respect to $\exists_{n+1}$ sentences.

By induction, $T^*\cup\diag{\str M}$ is complete.
\end{proof}

A model $\str M$ of a theory $T$ is \textbf{existentially closed}
if $T\cup\diag{\str M}$ is complete with respect to existential formulas.
For Theorem \ref{thm:Rob} then,
the proof of completeness of $T^*\cup\diag{\str M}$
is a generalization of the proof of the following.

\begin{porism}[Robinson's Test \protect{\cite[4.2.1, p.\ 92]{MR0153570}}]
A theory $T$ is model-complete,
provided that all of its models are existentially closed.
\end{porism}

For an inductive theory $T$
with a model $\str M$,
\begin{inparadesc}
\item[for every] system $\phi(\vec x,\vec z)$ in the signature of $T$,
\item[for all] $\vec a$ from $M$,
\item[for some] model $\str N$ of $T\cup\diag{\str M}$,
\end{inparadesc}
if $T\cup\diag{\str M}\cup\{\Exists{\vec z}\phi(\vec a,\vec z)\}$
is consistent, then $\phi(\vec a,\vec z)$ is solved in $\str N$.
By the method of the proof that every model of $T$
embeds in a model of $T^*$,
we can again put the last bold quantifier in front
and go on to obtain the following.

\begin{porism}
Every model of an inductive theory
embeds in an existentially closed model.
\end{porism}

The two porisms lead to the following result,
now standard \cite[3.5.15, p.\ 198]{MR91c:03026}.

\begin{theorem}[Eklof and Sabbagh \protect{\cite[7.10--3, pp.\ 286--8]{MR0277372}}]
An inductive theory $T$ has a model companion $T^*$
if and only if the models of $T^*$
are precisely the existentially closed models of $T$.
\end{theorem}

Robinson used Theorem \ref{thm:Rob}
to prove that the theory $\df_0$ of fields of characteristic $0$
with a derivation had a model-completion, $\dcf_0$.
But there are simpler practical approaches to obtaining model-completions;
simpler still, if all we want are model-companions.
First of all,
in the proof of Theorem \ref{thm:Rob},
we did not really need to extract the finite disjunction $\theta$
from all of the formulas $\theta_{(\str M,\vec a)}$.
Moreover, the proof that models of $T$ embed in models of $T^*$
did not require the formulas $\theta$ to be quantifier-free.
Neither is this required for the observation that the models of $T^*$
are the existentially closed models of $T$.
Thus we have the following.

\begin{porism}\label{por:Rob}\sloppy
For an inductive theory $T$ to have a model-completion,
a sufficient and necessary condition is that,
for every system $\phi(\vec x,\vec z)$
in the signature of $T$,
there is a set $\Theta$
of quantifier-free formulas $\theta(\vec x,\vec y)$ in that signature
such that, for all models $\str M$ of $T$,
for all tuples $\vec a$ of parameters from $M$
having the same length as $\vec x$,
the following conditions are equivalent:
\begin{compactenum}[(i)]
\item
$\phi(\vec a,\vec z)$ 
is soluble in some model of $T\cup\diag{\str M}$,
\item
$\theta(\vec a,\vec y)$ is soluble in $\str M$ itself
for some $\theta$ in $\Theta$.
\end{compactenum}
When such $\Theta$ do exist,
then the model-completion of $T$ 
is the theory $T^*$ axiomatized by the sentences \eqref{eqn:Rob},
where $\theta$ ranges over $\Theta$,
along with axioms for $T$ itself.
If the formulas in the sets $\Theta$ are not necessarily quantifier-free,
the theory $T^*$ is still the model-companion of $T$.
\end{porism}

\begin{sloppypar}
  Simpler axiomatizations than Robinson's for $\dcf_0$ 
  were found by showing that the axioms need not explicitly concern all systems
  \cite{MR0491149,MR99g:12006}.
  The general observation can be formulated as follows.
\end{sloppypar}

\begin{porism}\label{por:main}
Theorem \ref{thm:Rob} and its Porism \ref{por:Rob} still hold,
even if $\phi$ is constrained to range over a collection of systems containing,
\begin{compactenum}[1)]
\item
for each system $\psi(\vec x,\vec u)$ in the signature of $T$,
\item
for each model $\str M$ of $T$,
\item
for each tuple $\vec a$ of parameters from $M$,
\end{compactenum}
a system $\phi(\vec x,\vec u,\vec v)$
such that,
if $\Exists{\vec u}\psi(\vec a,\vec u)$
 is consistent with $T\cup\diag{\str M}$,
 then so is
 $\Exists{\vec u}\Exists{\vec v}\phi(\vec a,\vec u,\vec v)$,
 and
\begin{equation}\label{eqn:T-cup}
    T\cup\diag{\str M}\proves\phi(\vec a,\vec u,\vec v)\lto\psi(\vec a,\vec u).
\end{equation}
\end{porism}

We may refer to $\phi(\vec x,\vec u,\vec v)$
as a \textbf{refinement} of the system $\psi(\vec x,\vec u)$.
We shall apply Porism \ref{por:main} when $T$ is $\fav$
or, as in the next section, the theory of \emph{difference fields.}

\section{Difference fields}\label{sect:df}

A \textbf{difference field} is a field equipped with an automorphism.
As was observed in the introduction,
the theory of difference fields
in the signature $\{+,-,\times,0,1,{}\autom\}$
has a model-compan\-ion,
called $\acfa$.
Perhaps any axiomatization of $\acfa$ can be made to serve present purposes;
we shall derive the one that we shall use from Theorem \ref{thm:acfa}.
which is in the style of \cite[Thm~3.1, p.~1337]{MR2114160}.

The following lemma will be the reason for condition \eqref{eqn:I_0}
in Theorem \ref{thm:acfa}.
For notational economy,
our set $\upomega$ of natural numbers
is the set of von Neumann natural numbers,
where
\begin{equation*}
n=\{0,\dots,n-1\}=\{i\colon i<n\}.
\end{equation*}

\begin{fact}\label{fact:acfa}
  Suppose $(K,\sigma)$ is a difference field,
  %
and $I$ is a prime ideal of the polynomial ring 
$K[X_j\colon j<n]$,
and $\tau$ is an embedding of $m$ in $n$.
Write $X_j+I$ as $a_j$ whenever $j<n$.
For $\sigma$ to extend
to an automorphism of a field that includes $K[\vec a]$
so that $a_i{}\autom=a_{\tau(i)}$ whenever $i<m$,
it is necessary and sufficient that
\begin{equation*}
f(a_i\colon i<m)=0\iff f(a_{\tau(i)}\colon i<m)=0
\end{equation*}
for all $f$ in $K[X_i\colon i<m]$.
\end{fact}

\begin{theorem}\label{thm:acfa}
A difference-field $(K,\sigma)$
is existentially closed among all difference fields
if and only if,
\begin{compactenum}[1)]
\item
for all $m$ and $n$ in $\upomega$ such that $m\leq n$,
\item
for every injective function $\tau$ from $m$ into $n$,
\item
for every finite subset $I_0$ of $K[X_j\colon j<n]$,
\end{compactenum}
if $I_0$ generates a prime ideal $(I_0)$ of $K[X_j\colon j<n]$,
and
\begin{multline}\label{eqn:I_0}
\bigl\{f(X_{\tau(i)}\colon i<m)\colon f\in(I_0)\cap K[X_i\colon i<m]\bigr\}\\
=(I_0)\cap K[X_{\tau(i)}\colon i<m],
\end{multline}
then the system
\begin{equation}\label{eqn:df-sys}
\bigwedge_{f\in I_0}f=0\land\bigwedge_{i<m}X_i{}\autom=X_{\tau(i)}
\end{equation}
has a solution in $K$
(the case $m=0$ ensures that $K$ is algebraically closed).
\end{theorem}

\begin{proof}
We refine an arbitrary system of difference equations and inequations as follows.
Over a difference field $(K,\sigma)$,
suppose a system 
has a solution $(a_i\colon i<k)$ from some larger model.
Whenever $i<j<k$,
we may assume that $a_i\neq a_j$
and that the system has the inequation $X_i\neq X_j$
as one of its conjuncts.
We obtain a refinement having the form \eqref{eqn:df-sys}
as follows.
\begin{compactenum}
\item
In non-constant terms,
repeatly make the replacements
\begin{equation*}
\begin{array}{r*3{r}}
  \text{of}&       (t+u)\autom,&(-t)\autom,&(t      \cdot u)\autom\\
\text{with}&t\autom+  u \autom,&  -t\autom,& t\autom\cdot u \autom
\end{array}
\end{equation*}
respectively,
until $\sigma$ is applied only to variables and constant terms.
\item
For every atomic or negated atomic formula $\phi$
of the system that is not of the form $X\autom=Y$,
but in which $X\autom$ appears as an argument,
replace that argument with a new variable $Y$,
and introduce the new equation $X\autom=Y$.
\item
If for some $i$ less than $k$,
there is not already an equation of the form $X_i{}\autom=Y$,
then introduce such an equation,
$Y$ being a new variable.
\item
Replace any polynomial inequation $f\neq g$
with
\begin{equation*}
(f-g)\cdot X=1,
\end{equation*}
where $X$ is a new variable.
\end{compactenum}
After indexing the new variables appropriately,
we have that
\begin{compactenum}[1)]
\item
for some $m$ and $n$ in $\upomega$ such that $m\leq n$,
\item
for some function $\tau$ from $m$ into $n$,
\item
for some finite subset $I_0$ of $K[X_j\colon j<n]$,
\end{compactenum}
our system has the form of \eqref{eqn:df-sys}.
(It may be that some of the hidden parameters
are part of compound terms that involve $\sigma$;
but such terms can just be understood as standing for 
the appropriate elements of $K$.)
If $\tau$ is not injective,
then the system must have equations
$X_i{}\autom=X_{\ell}$ and $X_j{}\autom=X_{\ell}$,
where $k\leq i<j<m$;
but these equations imply $X_i=X_j$,
and so we can replace $X_j$ throughout with $X_i$.
Thus we may assume $\tau$ is injective.
In case the ideal generated by $I_0$ is not prime,
still,
for some $(a_i\colon k\leq i<n)$ in the larger difference field,
the new system has the solution $(a_i\colon i<n)$,
and we can then add enough equations $f(X_j\colon j<n)=0$
that are satisfied by $(a_i\colon i<n)$
so that $I_0$ becomes a set of generators of a prime ideal $\mathfrak P$,
and $(a_i\colon i<n)$ is a generic point over $K$ of the zero-set of $\mathfrak P$.
In this case \eqref{eqn:I_0} is satisfied.
For every solution $(b_i\colon i<n)$ of the latest system,
$(b_i\colon i<k)$ solves the original system.
Thus if $(K,\sigma)$ meets the given conditions,
it is existentially closed as a difference field.

Conversely,
under the given conditions, 
every system of the form \eqref{eqn:df-sys}
is indeed consistent with $(K,\sigma)$,
by Fact \ref{fact:acfa}:
if $\vec a$ is a generic zero of $I_0$,
we can extend $\sigma$ to an isomorphism from $K(a_i\colon i<m)$
to $K(a_{\tau(i)}\colon i<m)$,
and then to an automorphism of a field including $K(\vec a)$.
In this way, $\vec a$ solves \eqref{eqn:df-sys},
so this system must have a solution in $K$,
if $(K,\sigma)$ is existentially closed as a difference field.
\end{proof}

If we did not already know that $\acfa$ existed,
the foregoing theorem would prove it by Porism \ref{por:main},
since the conditions that \eqref{eqn:df-sys} must satisfy are first-order.
This is so,
because of the existence of appropriate bounds on degrees of polynomials,
as established in \cite{MR739626}.
In particular,
for all $n$ and $r$ in $\upomega$,
there are bounds $s$ and $t$ in $\upomega$
such that, 
for all fields $K$,
for all $m$ in $\upomega$,
for every ideal $I$ of $K[X_j\colon j<n]$
generated by a set $\{f_i\colon i<m\}$,
each $f_i$ having degree $r$ or less,
\begin{compactenum}[1)]
\item
the primeness of the ideal can be established by showing
\begin{equation*}
gh\in I\And g\notin I
\implies h\in I
\end{equation*}
for all polynomials $g$ and $h$ in $K[X_j\colon j<n]$
having degree $s$ or less, and
\item
membership in $I$ by polynomials like $gh$ having degree $2s$ or less
is established by polynomials of degree $t$ or less,
in the sense that,
if indeed $gh\in I$,
then $gh=\sum_{i<m}g_i\cdot f_i$
for some $g_i$ having degree $t$ or less.
\end{compactenum}
Because $\acf$ admits full elimination of quantifiers,
$\acfa$ is  the model-\emph{completion} 
of the theory of difference fields
that are algebraically closed as fields (compare to the last sentence in Porism \ref{por:Rob}).

\section{A model-completion}\label{sect:mod-comp}

We now consider the class of models $(K,\sigma,\vr)$ of $\fav$
such that
\begin{equation*}
  \Exists xx\notin\vr
\end{equation*}
and,
\begin{compactenum}[1)]
\item
for all $m$ and $n$ in $\upomega$ such that $m\leq n$,
\item
for every injective function $\tau$ from $m$ into $n$,
\item
for every finite subset $I_0$ of $\vr[X_j\colon j<n]$,
\item
for all subsets $\lambda$ of $n$ and $\kappa$ of $\lambda$,
\end{compactenum}
if
\begin{compactenum}[a)]
\item
$I_0$ generates a prime ideal $(I_0)$ of $K[X_j\colon j<n]$
such that the condition \eqref{eqn:I_0} in Theorem \ref{thm:acfa} holds, and
\item
when $S$ is the ring $\vr\bigl[I_0\cup\{X_{\ell}\colon\ell\in\lambda\}\bigr]$,
the ideal of $S$
generated by the set
$\maxi\cup I_0\cup\{X_k\colon k\in\kappa\}$
is proper, that is,
\begin{equation}\label{eqn:ideal}
\bigl(\maxi\cup I_0\cup\{X_k\colon k\in\kappa\}\bigr)S
\subsetneq S,
\end{equation}
\end{compactenum}
then $K$ contains a common solution to the system \eqref{eqn:df-sys} in Theorem \ref{thm:acfa}
and the system
\begin{equation}\label{eqn:cond}
\bigwedge_{\ell\in\lambda}X_{\ell}\in\vr
\land\bigwedge_{k\in\kappa}X_k\in\maxi.
\end{equation}
The case $m=0=\lambda$
ensures that $K$ is algebraically closed.

As the existentially closed difference-fields,
characterized by Theorem \ref{thm:acfa},
are just the models of a certain theory $\acfa$,
so the models of $\fav$ just described
are the models of a certain theory,
which we shall call $\fav^*$.
In particular,
the condition \eqref{eqn:ideal} is first-order.
Indeed, this condition means
there are no $g_f$ and $h_k$ in $S$ such that
\begin{equation}\label{eqn:gfhk}
\sum_{f\in I_0}g_f\cdot f+\sum_{k\in\kappa}h_k\cdot X_k\equiv 1\pmod{\maxi}.
\end{equation}
The ring $S$ is, for some subset $I_1$ of $I_0$,
isomorphic to the quotient of the polynomial ring
$\vr[\{Y_f\colon f\in I_1\}\cup\{X_{\ell}\colon\ell\in\lambda\}]$
by an element of bounded degree.
We can also work over the residue field $\vr/\maxi$, instead of $\vr$.
Thus, by \cite{MR739626},
for we can bound the degrees
of the $g_f$ and $h_k$ that would make \eqref{eqn:gfhk} true.

\begin{lemma}\label{lem:emb}
Every model of $\fav$ embeds in a model of $\fav^*$.
\end{lemma}

\begin{proof}
Let $(K,\sigma,\vr)$ be a model of $\fav$
such that,
\begin{compactenum}[1)]
\item
for some $m$ and $n$ in $\upomega$, where $m\leq n$,
\item
for some injective function $\tau$ from $m$ into $n$,
\item
for some finite subset $I_0$ of $\vr[X_j\colon j<n]$,
\item
for some sets $\lambda$ and $\kappa$, where $\kappa\included\lambda\included n$,
\end{compactenum}
we have that
\begin{compactenum}[a)]
\item
$I_0$ generates a prime ideal $(I_0)$ of $K[X_j\colon j<n]$
such that the condition \eqref{eqn:I_0} in Theorem \ref{thm:acfa} holds, and
\item
when $S$ is the ring $\vr\bigl[I_0\cup\{X_{\ell}\colon\ell\in\lambda\}\bigr]$,
then \eqref{eqn:ideal} holds.
\end{compactenum}
We already know,
as in the proof of Theorem \ref{thm:acfa},
that the system \eqref{eqn:df-sys}
has a solution $\vec a$
in a difference field $(L,\widetilde{\sigma})$
of which $(K,\sigma)$ is a substructure;
and we may require $\vec a$ to be a generic solution of the field-theoretic part
\begin{equation*}
\bigwedge_{f\in I_0}f=0
\end{equation*}
of the system.
We now show that $L$ has a valuation ring $\widetilde{\vr}$
such that
\begin{equation*}
K\cap\widetilde{\vr}=\vr
\end{equation*}
and $\vec a$ solves \eqref{eqn:cond}, that is,
\begin{equation}\label{eqn:a}
\bigwedge_{\ell\in\lambda}a_{\ell}\in\widetilde{\vr}
\land\bigwedge_{k\in\kappa}a_k\in\widetilde{\maxi}.
\end{equation}
We can do this by refining the proof
of Chevalley's theorem on extending valuations
(for which see \cite[Thm 3.1.1, p.\ 57]{Engler-Prestel}),
or simply by using a refinement
\cite[Thm 5, p.\ 12]{MR0120249}
of the theorem itself.
By this refinement,
the sub-ring $S=\vr[I_0\cup\{X_{\ell}\colon\ell\in\lambda\}]$ of $K(X_j\colon j<n)$
has a prime ideal $\mathfrak P$ that includes the proper ideal
generated by $\maxi\cup I_0\cup\{X_k\colon k\in\kappa\}$;
therefore some valuation ring $\vr^*$ of $K(X_j\colon j<n)$
with maximal ideal $\maxi^*$ satisfies
\begin{align*}
  S&\included\vr^*,&
  \mathfrak P&=\maxi^*\cap S.
\end{align*}
In particular,
\begin{align*}
	\{X_{\ell}\colon\ell\in\lambda\}&\included\starred{\vr},&
	I_0\cup\{X_k\colon k\in\kappa\}&\included\starred{\maxi},&
        \maxi^*\cap\vr&=\maxi.
\end{align*}
Now we can understand $\starred{\vr}/(I_0){\starred{\vr}}$
as a valuation ring of $K(\vec a)$.
By Chevalley's Theorem,
we can extend this valuation ring
to a valuation ring $\widetilde{\vr}$ of $L$.
In this case \eqref{eqn:a} holds.
\end{proof}

\begin{lemma}\label{lem:mc}
	$\fav^*$ is model complete.
\end{lemma}

\begin{proof}
We proceed as in the proof of Theorem \ref{thm:acfa}.
Over a model of $\fav$,
supposing a system of atomic and negated atomic formulas
has a solution $(a_i\colon i<k)$ from some larger model,
we transform the system into the conjunction of a system of the form \eqref{eqn:df-sys}
and a system of the form \eqref{eqn:cond}.
We proceed as before,
but now, since formulas $f\in\vr$ and $f\in\maxi$ and their negations may appear,
we can eliminate negations by applying \eqref{eqn:notin-m} and \eqref{eqn:notin-o},
and we can replace $f\in\vr$ and $f\in\maxi$ themselves with $X\in\vr\land f=X$
and $X\in\maxi\land f=X$ respectively,
where $X$ is a new variable.
\end{proof}

\begin{theorem}\label{thm:main}
$\fav^*$ is the model companion of $\fav$
and the model completion of the theory of models of $\fav$
whose fields are algebraically closed.
\end{theorem}

\begin{proof}
The first part is the content of Lemmas \ref{lem:emb} and \ref{lem:mc}.
When the underlying field is required to be algebraically closed,
then, by quantifier-elimination in the theory of such fields,
the conditions that the systems \eqref{eqn:df-sys} and \eqref{eqn:cond}
are to meet are given by a quantifier-free formula.
\end{proof}

\begin{theorem}\label{thm:opt}
The theory of models of $\acfa$ that also have valuations
is not the model companion of $\fav$.
\end{theorem}

\begin{proof}
We show that there is a model $(K,\sigma,\vr)$
of $\fav$ that is not a model of $\fav^*$,
although the reduct $(K,\sigma)$ is a model of $\acfa$.

It is known \cite{2004math......6514H}
that every nonprincipal ultraproduct
of the algebraic closures of the fields of prime order,
each equipped with its Frobenius automorphism,
is a model of $\acfa$.
Now let
\begin{equation*}
(K,\sigma,\vr)=\prod_p\bigl(\F_p(T)\alg,x\mapsto x^p,\vr_T\bigr)/\mathscr U
\end{equation*}
for some nonprincipal ultrafilter $\mathscr U$ on the set of primes
and some valuation ring $\vr_T$ of each $\F_p(T)$.
For example,
$\vr_T$ might be the $T$-adic valuation ring,
consisting of those elements of $\F_p(T)$ that, considered as functions of $T$,
are well-defined at $0$.
The structure is as desired since 
by \eqref{eqn:val+} it satisfies
\begin{equation*}
 \Forall x\bigl(\val(x)>0\lto\val(x\autom)>0\bigr),
\end{equation*}
that is,
$\Forall x\bigl(x\in\maxi\lto x\autom\in\maxi\bigr)$,
while in every model of $\fav^*$ the system
\begin{equation*}
x\in\maxi\land x\autom\notin\maxi
\end{equation*}
is soluble.
\end{proof}


\begin{thebibliography}{10}

\bibitem{MR2725200}
Salih Azgin.
\newblock Valued fields with contractive automorphism and {K}aplansky fields.
\newblock {\em J. Algebra}, 324(10):2757--2785, 2010.

\bibitem{MR2749570}
Salih Azgin and Lou van~den Dries.
\newblock Elementary theory of valued fields with a valuation-preserving
  automorphism.
\newblock {\em J. Inst. Math. Jussieu}, 10(1):1--35, 2011.

\bibitem{MR2325101}
Luc B{\'e}lair, Angus Macintyre, and Thomas Scanlon.
\newblock Model theory of the {F}robenius on the {W}itt vectors.
\newblock {\em Amer. J. Math.}, 129(3):665--721, 2007.

\bibitem{MR0491149}
Lenore Blum.
\newblock Differentially closed fields: a model-theoretic tour.
\newblock In {\em Contributions to algebra (collection of papers dedicated to
  {E}llis {K}olchin)}, pages 37--61. Academic Press, New York, 1977.

\bibitem{MR91c:03026}
Chen~Chung Chang and H.~Jerome Keisler.
\newblock {\em Model theory}, volume~73 of {\em Studies in Logic and the
  Foundations of Mathematics}.
\newblock North-Holland Publishing Co., Amsterdam, third edition, 1990.
\newblock First edition 1973.

\bibitem{MR2000f:03109}
Zo{\'e} Chatzidakis and Ehud Hrushovski.
\newblock Model theory of difference fields.
\newblock {\em Trans. Amer. Math. Soc.}, 351(8):2997--3071, 1999.

\bibitem{MR3397448}
Salih Durhan and G{\"o}nen{\c{c}} Onay.
\newblock Quantifier elimination for valued fields equipped with an
  automorphism.
\newblock {\em Selecta Math. (N.S.)}, 21(4):1177--1201, 2015.

\bibitem{MR0277372}
Paul Eklof and Gabriel Sabbagh.
\newblock Model-completions and modules.
\newblock {\em Ann. Math. Logic}, 2(3):251--295, 1970/1971.

\bibitem{Engler-Prestel}
Antonio~J. Engler and Alexander Prestel.
\newblock {\em Valued Fields}.
\newblock Springer, Berlin, 2005.

\bibitem{MR94e:03002}
Wilfrid Hodges.
\newblock {\em Model theory}, volume~42 of {\em Encyclopedia of Mathematics and
  its Applications}.
\newblock Cambridge University Press, Cambridge, 1993.

\bibitem{2004math......6514H}
E.~{Hrushovski}.
\newblock {The Elementary Theory of the Frobenius Automorphisms}.
\newblock {\em ArXiv Mathematics e-prints}, June 2004.
\newblock arXiv:math/0406514 [math.LO].

\bibitem{MR1791373}
Hirotaka Kikyo.
\newblock Model companions of theories with an automorphism.
\newblock {\em J. Symbolic Logic}, 65(3):1215--1222, 2000.

\bibitem{MR2138334}
Hirotaka Kikyo.
\newblock On generic predicates and the amalgamation property for
  automorphisms.
\newblock {\em Proc. Sch. Sci. Tokai Univ.}, 40:19--24, 2005.

\bibitem{MR99c:03046}
Angus Macintyre.
\newblock Generic automorphisms of fields.
\newblock {\em Ann. Pure Appl. Logic}, 88(2-3):165--180, 1997.
\newblock Joint AILA-KGS Model Theory Meeting (Florence, 1995).

\bibitem{MR2963021}
Koushik Pal.
\newblock Multiplicative valued difference fields.
\newblock {\em J. Symbolic Logic}, 77(2):545--579, 2012.

\bibitem{MR2114160}
David Pierce.
\newblock Geometric characterizations of existentially closed fields with
  operators.
\newblock {\em Illinois J. Math.}, 48(4):1321--1343, 2004.

\bibitem{MR2505433}
David Pierce.
\newblock Model-theory of vector-spaces over unspecified fields.
\newblock {\em Arch. Math. Logic}, 48(5):421--436, 2009.

\bibitem{MR3226008}
David Pierce.
\newblock Fields with several commuting derivations.
\newblock {\em J. Symb. Log.}, 79(1):1--19, 2014.

\bibitem{MR99g:12006}
David Pierce and Anand Pillay.
\newblock A note on the axioms for differentially closed fields of
  characteristic zero.
\newblock {\em J. Algebra}, 204(1):108--115, 1998.

\bibitem{MR0153570}
Abraham Robinson.
\newblock {\em Introduction to model theory and to the metamathematics of
  algebra}.
\newblock North-Holland Publishing Co., Amsterdam, 1963.

\bibitem{MR2002d:03066}
Thomas Scanlon.
\newblock A model complete theory of valued ${D}$-fields.
\newblock {\em J. Symbolic Logic}, 65(4):1758--1784, 2000.

\bibitem{MR739626}
L.~van~den Dries and K.~Schmidt.
\newblock Bounds in the theory of polynomial rings over fields. {A} nonstandard
  approach.
\newblock {\em Invent. Math.}, 76(1):77--91, 1984.

\bibitem{MR0120249}
Oscar Zariski and Pierre Samuel.
\newblock {\em Commutative Algebra. {V}ol. {II}}.
\newblock The University Series in Higher Mathematics. D. Van Nostrand Co.,
  Inc., Princeton, N. J., Toronto, London, New York, 1960.

\end{thebibliography}

\end{document}